\documentclass[11pt, a4paper, twoside]{article}

\usepackage[ngerman,english]{babel}	

\usepackage{titling}
\newcommand{\subtitle}[1]{%
  \posttitle{%
    \par\end{center}
    \begin{center}\large#1\end{center}
    \vskip0.5em}%
}

\usepackage{graphicx}
\usepackage{amsmath} 
\usepackage{amssymb} 
\usepackage{amsfonts}
\usepackage{amsthm}
\usepackage{enumerate}
\usepackage{makeidx} 
\usepackage{verbatim}
\usepackage{epsfig}
\usepackage{hyperref} 
\usepackage{dsfont}
\usepackage{wasysym} 
\usepackage{caption}	
\usepackage{float}
\usepackage{stackrel}	
\usepackage{nicefrac}
\usepackage{todonotes}
\usepackage{fancyhdr}
\usepackage{amsfonts}
\usepackage{amsmath}
\usepackage{amssymb}
\usepackage{amsthm}
\usepackage{tikz}
\usepackage{amscd}
\usepackage{amstext}
\usepackage{graphics}
\usepackage{graphicx}
\usepackage{authblk}

\usepackage{hyperref} 
\hypersetup{colorlinks=true,
						linkcolor=black,
						citecolor=black,
						urlcolor=blue}

\makeindex

\makeindex      
     
\newlength{\fixboxwidth}     
\newcommand{\R} {{\mathbb R}}
\newcommand{\C} {{\mathbb C}}
\newcommand{\K} {{\mathbb K}}  
\newcommand{\N} {{\mathbb N}}
\newcommand{\Z} {{\mathbb Z}}
   
\newcommand{\W} {{\bf W }}

\newcommand{\tord} {{\mathbb \Torus^d}}	
\newcommand{\Torus}{{\mathbb T}}
\renewcommand{\P}{\mathbb P}

\DeclareMathOperator\expect{\mathbb{E}}
\newcommand{\Hilbert}{{\mathcal H}}

\newcommand{\euler}{{\textup e}}
\newcommand{\imag}{{\textup i}}
\renewcommand{\Re}{\operatorname{Re}}
\renewcommand{\Im}{\operatorname{Im}}
\newcommand{\embed}{\hookrightarrow}

\newcommand\Sphere{\mathbb{S}}

\DeclareMathOperator\card{card}

\newcommand{\wt}{\widetilde }      

\DeclareMathOperator\id{id}

\renewcommand{\epsilon}{\varepsilon}      
\newcommand{\ran}{{\rm ran }}     
\newcommand{\deter}{{\rm det }}

\newcommand{\lin}{{\rm lin}}
\newcommand{\nonlin}{{\rm nonlin}}
\newcommand{\dint}{\,\mathrm{d}} 
\newcommand\veck{\mathbf{k}}
\newcommand\vecs{\mathbf{s}}
\newcommand\vecx{\mathbf{x}}
\newcommand\vecy{\mathbf{y}}

\newcommand{\vecxi}{\boldsymbol{\xi}}				
\newcommand{\vecnu}{\boldsymbol{\nu}}				
\newcommand{\veczeta}{\boldsymbol{\zeta}}		

\usepackage{microtype} 

\usepackage{bookmark}
\makeatletter 
  \providecommand*{\toclevel@author}{999}
  \providecommand*{\toclevel@title}{0}
\makeatother

\theoremstyle{plain}       
     
\newtheorem{theorem}{Theorem}[section]      
\newtheorem{corollary}[theorem]{Corollary}      
\newtheorem{lemma}[theorem]{Lemma}     
\newtheorem{proposition}[theorem]{Proposition}      
     
\theoremstyle{definition}     
     
\newtheorem{definition}[theorem]{Definition}
\newtheorem*{notation*}{Notation}
     
\newtheorem{remark}[theorem]{Remark}     
\newtheorem{algorithm}[theorem]{Algorithm}

\numberwithin{equation}{section}   
\begin{document}     

\title{Monte Carlo Methods for Uniform Approximation on\\
Periodic Sobolev Spaces with Mixed Smoothness}


\author[a]{Glenn Byrenheid\thanks{E-mail: byrenheid.glenn@gmail.com}}
\author[b]{Robert J.\ Kunsch\thanks{E-mail: robert.kunsch@uni-osnabrueck.de}}
\author[c]{Van Kien Nguyen\thanks{E-mail: kiennv@utc.edu.vn \& kien.nv.hp@gmail.com}}
\affil[a]{Hausdorff-Center for Mathematics, 
	Endenicher Allee~62, 53115 Bonn, Germany}
\affil[b]{Universit\"at Osnabr\"uck, Institut f\"ur Mathematik,
	Albrechtstr.~28a, 49076 Osnabr\"uck, Germany}
\affil[c]{Department of Mathematics, University of Transport and Communications,
					No.3 Cau Giay Street, Lang Thuong Ward, Dong Da District,
					Hanoi, Vietnam}

\date{\today}
\maketitle
     

\begin{abstract}
	We consider the order of convergence for
	linear and nonlinear Monte Carlo approximation of compact embeddings
	from Sobolev spaces of dominating mixed smoothness
	defined on the torus~\mbox{$\tord$}
	into the space~\mbox{$L_{\infty}(\tord)$}
	via methods that use arbitrary linear information.
	These cases are interesting because
	we can gain a speedup
	of up to~$1/2$ in the main rate compared to the
	worst case approximation.
	In doing so we determine the rate for some cases that have been
	left open by Fang and Duan.
\end{abstract}


\section{Introduction}

Nowadays, Monte Carlo methods are widely used in many areas of applied mathematics.
Especially for the computation of integrals,
randomization will usually speed up the order of convergence
compared to deterministic methods.
It is well known that for certain function approximation problems
Monte Carlo helps in a similar way.
This is basically due to the fundamental work
of Math\'e 1991~\cite{Ma91} and Heinrich 1992~\cite{He92}.

Function spaces of dominating mixed smoothness were introduced
by S.M.~Nikol'skij in the early 1960s. 
Recently, there is an increasing interest in information-based complexity
and high-dimensional approximation in these spaces.
Function spaces of this type also play an important role in many real-world problems. 
For example, there exist a number of problems in finance and quantum chemistry
modelled on function spaces of dominating mixed smoothness. 
We refer to the monographs \cite{Glas-04,Yse-10}. 

Let $\tord:=[0,1)^d$ be the $d$-dimensional torus
and $\W_p^r(\tord)$ be the periodic Sobolev spaces
of dominating mixed smoothness~$r$ on $\tord$.
In this paper we study $L_\infty$-approximation of the class $\W_p^r(\tord)$
where we supplement the results of 
Fang and Duan~\cite{FD07,FD08} on $L_q$-approximation. 
Let us mention that the study of $L_\infty$-approximation
of function spaces with dominating mixed smoothness
is much harder compared to the case \mbox{$1 < q < \infty$}
and different techniques are needed,
see comments and open problems in \cite[Section~4.6]{DTU17}. 
Moreover, we hope that the way we present the algorithm here
will illuminate the nature of randomized approximation via linear information
towards a better understanding of general $L_q$-approximation as well.

Let $e^{\deter,\lin}(n,S)$ and $e^{\ran,\lin}(n,S)$
denote the minimal deterministic and randomized errors
for linear approximation of the operator~$S$
if we use $n$~deterministic or randomized information operations
from the class of all linear functionals,
respectively.
By using an estimate on the expected
norm of random trigonometric polynomials we can bound the order of convergence
for linear Monte Carlo approximation,
	\begin{equation*}
			\left(\frac{ (\log n)^{(d-1)}}{n}\right)^{r-(\frac{1}{p}-\frac{1}{2})_+}
				\preceq \, e^{\ran,\lin}\left(n,\W^r_p(\tord) \embed L_{\infty}(\tord)\right)
				\, \preceq \, \left(\frac{ (\log n)^{(d-1)}}{n}
														\right)^{r-(\frac{1}{p}-\frac{1}{2})_+}
															\sqrt{\log n}\,.
	\end{equation*}
Comparing our result with the already known result on
deterministic approximation,
\begin{equation*}
	e^{\deter,\lin}\left(n, \W^r_2(\tord) \embed L_{\infty}(\tord)\right)
		\,\asymp\, \frac{(\log n)^{r(d-1)}
								}{n^{r-1/2}} \,,
\end{equation*}
see Temlyakov~\cite{Tem93},
we observe that randomization improves the order of convergence
by a factor up to~$n^{1/2}$ when~$p = 2$.

The study of Monte Carlo methods for $L_\infty$-approximation
is particularly interesting
since in the deterministic setting
linear methods are optimal,
see~\cite[Theorem~4.5 and 4.8]{NW08}.
In the randomized setting the situation is different.
In combination with known results
on nonlinear deterministic methods for $L_2$-approximation
we shall show that the optimal Monte Carlo approximation rate
for spaces~\mbox{$\W_p^r(\tord)$} with~\mbox{$1<p<2$}
is better than what can be achieved with merely linear methods.
More precisely, we prove 
	\begin{equation*}
			\left(\frac{ (\log n)^{(d-1)}}{n}\right)^r
				\ \preceq\ e^{\ran,\nonlin}
											\left(n,\W^r_p(\tord)
																\embed L_{\infty}(\tord)
											\right)
				\ \preceq\ \left(\frac{ (\log n)^{(d-1)}}{n}
														\right)^r
															\sqrt{\log n} \,.
	\end{equation*}

The paper is organized as follows.
In the second section we shall recall some definitions
from information-based complexity
and give basic properties of error notions.
The next section is devoted to
a fundamental Monte Carlo function approximation method in an abstract setting,
which goes back to Math\'e 1991~\cite{Ma91}.
Our main results for spaces of dominating mixed smoothness
are discussed and proven in Section~\ref{sec:main}.

\begin{notation*}
	For a real number~$a$ we put $a_+ := \max\{a,0\}$.
	By $\lfloor a \rfloor$ we denote the integer part of~$a$.
	The notion $a_n \preceq b_n$ for sequences~$(a_n)_{n \in \N}$
	and $(b_n)_{n \in \N}$ means that
	there exists a constant $C>0$ such that $a_n \leq C \, b_n$
	for sufficiently large~$n$.
	The symbol~\mbox{$a_n \asymp b_n$} will be used as an abbreviation for
	\mbox{$a_n \preceq b_n \preceq a_n$}.
	Note that the implicit constants in this paper
	may depend on the dimension~$d$ and the parameters~$p$ and $r$.
	The cardinality of a finite set~$Q$ is denoted by~$|Q|$.
	For multi-indices~$\veck=(k_1,... , k_d) \in \Z^d$ we put
	\mbox{$|\veck|_1 := |k_1| + \ldots + |k_d|$}.
	We write~\mbox{$(\veck,\vecx) := k_1 x_1 + \ldots + k_d x_d$}
	for the standard scalar product
	with torus elements~\mbox{$\vecx \in \tord = [0,1)^d$}.
	The inner product~\mbox{$\langle \cdot, \cdot \rangle_{\Hilbert}$}
	within a Hilbert space~$\Hilbert$ shall be semilinear in the first argument
	in the complex setting.
	Finally, the embedding operator for Banach spaces~$F \subset G$
	will be denoted by~$F \embed G$.
\end{notation*}

\section{Types of errors and Monte Carlo approximation}

We will give a short introduction to the notions from \emph{information-based complexity}.
For more details on different error and algorithmic settings see,
for example, the books~\cite{NW08,TWW88}.

Let $S: F \to G$ be a bounded
linear operator between Banach spaces, 
the so-called \emph{solution operator}. We aim to approximate~$S$
for \emph{inputs}~$f \in F$ with respect to the norm of
the \emph{target space}~$G$.
The vector spaces involved may be real or complex, \mbox{$\K \in \{\R,\C\}$}.
(The field under consideration affects the class of admissible algorithms,
even if only real-valued functions are approximated.)

Let~$(\Omega,\Sigma,\mathbb{P})$ be a suitable probability space.
Further, let $\mathcal{B}(F)$ and $\mathcal{B}(G)$ denote the Borel
\mbox{$\sigma$-algebra} of $F$ and~$G$, respectively.
By \emph{randomized algorithms}, also called \emph{Monte Carlo algorithms},
we understand
\mbox{$\Sigma \otimes \mathcal{B}(F)
	-\mathcal{B}(G)$}-measurable
mappings
$A_n = (A_n^{\omega}(\cdot))_{\omega \in \Omega}:
	\Omega \times F \to G$.
This means that the output $A_n(f)$ for an input~$f$ is random, depending
on~\mbox{$\omega \in \Omega$}.
We consider algorithms of cardinality~$n$ that use at most
$n$~\emph{continuous linear functionals} as information,
i.e.~\mbox{$A_n^{\omega} = \phi^{\omega} \circ N^{\omega}$} where
$N^{\omega} : F \to \K^n$ is the so-called
\textit{information mapping}.
The mapping $\phi^{\omega}: \K^n \to G$ generates an output
$g = \phi^{\omega}(\vecy) \in G$ as
a compromise for all possible inputs that lead to
the same information~\mbox{$\vecy = N^{\omega}(f) \in \K^n$}.
In this paper we only consider \emph{non-adaptive} information mappings of the shape
\begin{equation}
	N^{\omega}(f)
	\,=\, [L_1^{\omega}(f),\ldots,L_n^{\omega}(f)]
	\,=\, (y_1,\ldots,y_n)
	\,=\, \vecy \,,
\end{equation}
where all functionals~$L_k^{\omega}$ are chosen at once.
This is equivalent to that $N^{\omega}$~is a linear mapping for any
fixed random element~\mbox{$\omega \in \Omega$}.
If the mapping~$\phi^{\omega}$ is linear as well,
the Monte Carlo method~$A_n$ is called \emph{linear} itself.
By $\mathcal{A}_n^{\ran,\lin}$ we denote the class of all
linear Monte Carlo algorithms that use $n$~pieces information,
the broader class of nonlinear algorithms is
denoted~$\mathcal{A}_n^{\ran,\nonlin}$.
We regard the class of \emph{deterministic algorithms} as a
subclass~$\mathcal{A}_n^{\rm det,\star} \subset \mathcal{A}_n^{\rm ran,\star}$
($\star \in \{\lin,\nonlin\}$)
of algorithms that are independent from~\mbox{$\omega \in \Omega$}
(this means in particular that we assume deterministic algorithms to be measurable),
for a particular algorithm we write \mbox{$A_n = \phi \circ N$}, omitting~\mbox{$\omega$}.

For a deterministic algorithm~$A_n$ the (absolute) \emph{error~at~$f$} is
defined as the distance between output and exact solution,
\begin{equation}
	e(A_n,f) \,:=\, \|S f - A_n(f)\|_G \,.
\end{equation}
For randomized algorithms~$A_n = (A_n^{\omega}(\cdot))$ this
can be generalized as the \emph{expected error at~$f$},
\begin{equation}
	e(A_n,f) \,:=\, \expect \|S f - A_n^{\omega}(f)\|_G \,.
\end{equation}
(The expectation~$\expect$ is written for the integration over
all~$\omega \in \Omega$ with respect to~$\P$.)

The \emph{global error} of an algorithm~$A_n$ is defined as the error for the
worst input from the unit ball of~$F$, we write
\begin{equation} \label{eq:errglobal}
	e(A_n,S) \,:=\, \sup_{\|f\|_F \leq 1} e(A_n,f) \,.
\end{equation}
We define the \emph{$n$-th minimal error} of a problem
as the error of optimal algorithms,
\begin{equation*}
	e^{\diamond,\star}(n,S)
		\,:=\, \inf_{A_n \in \mathcal{A}^{\diamond,\star}} e(A_n,S) \,,
\end{equation*}
where $\diamond \in \{\deter,\ran\}$
and $\star \in \{\lin,\nonlin\}$.
These quantities are inherent properties of the problem~$S$, so
$e^{\ran,\star}(n,S)$ is called the \emph{Monte Carlo error},
$e^{\det,\star}(n,S)$ the \emph{worst case error}
of the problem~$S$.

Obviously, these error quantities are decreasing (or steady)
for growing~$n$.
By definition~\eqref{eq:errglobal},
restricting the set of inputs will diminish the error.
Similarly, a weaker norm for the target space has the same effect.
Due to homogeneity, equivalent norms will still give the same speed of convergence.
In general, broader classes of algorithms can only lead to a larger error,
so, since randomization and nonlinearity are additional features for algorithms,
we have
\begin{equation}
	e^{\ran,\star}(n,S) \,\leq\, e^{\deter,\star}(n,S)
		\quad \text{and} \quad
	e^{\diamond,\nonlin}(n,S) \,\leq\, e^{\diamond, \lin}(n,S)\,.
\end{equation}

\begin{remark}[Adaption and varying cardinality]
	All known algorithms used for linear function approximation problems
	(as we consider them in this paper)
	are non-adaptive and with fixed cardinality~$n$.
	Lower bounds, however, should also hold for more general algorithmic settings.
	This could be \emph{adaptive} information $N^{\omega}$, where the choice of the
	functionals may depend on previously obtained information.
	We could also consider algorithms~$A$ with \emph{varying cardinality},
	where the number~$n = n(f,\omega)$ of information operations
	may be randomized and adaptively depend on the input,
	the cardinality of~$A$ is then defined by the average cost,
	\mbox{$\card(A) := \sup_{f} \expect n(f,\omega)$}.
	
	Linear methods are always nonadaptive,
	hence varying cardinality means~\mbox{$n = n(\omega)$}.
	In this case,
	the error~\mbox{$\wt{e}^{\ran,\lin}(n)$} for algorithms with varying cardinality,
	and~\mbox{$e^{\ran,\lin}(n)$} for fixed cardinality, are closely related,
	\begin{equation*}
		{\textstyle \frac{1}{2}} \, e^{\ran,\lin}(2n)
			\,\leq\, \wt{e}^{\ran,\lin}(n)
			\,\leq\, e^{\ran,\lin}(n) \,,
	\end{equation*}
	see Heinrich~\cite[p.~289/290]{He92}.
	Consequently, varying cardinality does not affect the rate of approximation.
	
	For the nonlinear setting,
	the lower bounds in Fang and Duan~\cite{FD07} and the present paper
	rely on a technique due to Heinrich~\cite{He92},
	which is based on norm expectations for Gaussian measures.
	These lower bounds hold for adaptive algorithms as well.
	In Heinrich's original paper non-adaptively varying cardinality~\mbox{$n = n(\omega)$}
	is taken into account,
	in~\cite[Section~2.3.1]{Ku17} it is shown how lower bounds of this type
	extend to adaptively varying cardinality~\mbox{$n = n(\omega,f)$}.
\end{remark}

\section{A fundamental Monte Carlo function approximation method}

The following Proposition originates from Math{\'e} 1991~\cite[Lemma~5]{Ma91}
and is a key component for the Monte Carlo approximation
of Hilbert space functions.
Here we keep it a little more general than in the original paper,
where the target space~$G$ was restricted to sequence spaces~$\ell_q^m$,
\mbox{$2 < q \leq \infty$}.
The proof is included for completeness.

\begin{proposition}\label{prop:Ma91_l2G}
	Let $S:\ell_2^m = \R^m \to G$ be a linear operator.
	For~\mbox{$n < m$} let the information mapping~$N$
	be a random \mbox{$(n \times m$)}-Matrix with entries
	$N(i,j) = \frac{1}{\sqrt{n}} \, \xi_{ij}$,
	where the $\xi_{ij}$ are
	independent standard Gaussian random variables.
	Then \mbox{$A_n := S \, N^{\top} N$}
	defines a linear rank-$n$ Monte Carlo method and its error
	is bounded from above by
	\begin{equation*}
		e(A_n,S: \ell_2^m = \R^m \to G)
			\,\leq\, \frac{2}{\sqrt{n}} \, \expect \|S \vecxi\|_G
	\end{equation*}
	for $n<m$ where $\vecxi$ is a standard Gaussian vector in $\R^m$.
\end{proposition}

\begin{proof}
	Note that~$A_n$ is an unbiased linear Monte Carlo algorithm.
	To see this, take an input vector \mbox{$\vecx = (x_1,\ldots,x_m) \in \R^m$},
	then
	\begin{equation*}
		\expect (N^{\top} N \vecx)(i)
			\,=\, \frac{1}{n} \sum_{j=1}^n \sum_{k=1}^m
					\underbrace{\expect \xi_{ji} \xi_{jk}}_{= \delta_{ik}} \, x_k
			\,=\, x_i \,,
	\end{equation*}
	that is $\expect N^{\top} N \vecx = \vecx$,
	and by linearity of~$S$ we have $\expect A_n \vecx = S \vecx$.

	We start from the definition of the error for an input~$\vecx \in \ell_2^m$,
	\begin{align}
		e(A_n,\vecx)
			&\,=\, \expect\|S \vecx - S N^{\top} N \vecx\|_G \,.
			\nonumber
		\intertext{%
	Now, let~$M$~be an independent copy of~$N$.
	We write~$\expect^{\prime}$ for expectations with respect to $M$,
	and $\expect$ with respect to $N$.
	Using $\expect^{\prime} M^{\top}M = \id_{\R^m}$ and $\expect^{\prime} M = 0$,
	we can write} 
			&\,=\, \expect \| \expect^{\prime} S(M^{\top}M - M^{\top}N
																		+ N^{\top} M - N^{\top} N) \, \vecx \|_G
			\nonumber\\
			& 
	 \stackrel{\text{$\Delta$-ineq.}}{\leq}
			2 \expect \expect^{\prime}
									\bigg\| S \Big(\frac{M+N}{\sqrt{2}}\Big)^{\top}
													\Big(\frac{M-N}{\sqrt{2}}\Big) \, \vecx \bigg\|_G \,.
		\label{eq:tri}
		\intertext{%
	The distribution of $(M,N)$ is identical to that
	of $\left(\frac{M+N}{\sqrt{2}},\frac{M-N}{\sqrt{2}}\right)$,
	therefore we rewrite the above term as}
			&\,=\, 2 \expect \expect^{\prime} \|S N^{\top} M \vecx \|_G.
		\nonumber
		\intertext{%
	Here, $M \vecx$ is a Gaussian vector distributed
	like~$\frac{\|\vecx\|_2}{\sqrt{n}} \, \veczeta$ with $\veczeta$~being a
	standard Gaussian vector on $\R^n$.
	So we continue,
	$\expect^{\prime}$ now denoting the expectation with respect to~$\veczeta$,}
			&\,=\, \frac{2 \|\vecx\|_2}{\sqrt{n}} \,
					\expect \expect^{\prime} \|S N^{\top} \veczeta \|_G.
			\nonumber
		\intertext{%
	For fixed $\veczeta$, the distribution of $N^{\top} \veczeta$ is identical to that of
	$\frac{\|\veczeta\|_2}{\sqrt{n}} \, \vecxi$
	where $\vecxi$~is a standard Gaussian vector on $\R^m$,
	we write $\expect$ for the expectation with respect to~$\vecxi$.
	By Fubnini's theorem we get}
			&\,=\, \frac{2 \|\vecx\|_2}{n}
					\, \expect^{\prime} (\|\veczeta\|_2 \, \expect \|S \vecxi\|_G)\,.
		\nonumber
		\intertext{%
	Using
	\mbox{$\expect^{\prime} \|\veczeta\|_2
						\leq \sqrt{\expect^{\prime} \|\veczeta\|_2^2}
						= \sqrt{n}$},
	we finally obtain}
		e(A_n,x) &\,\leq\, \frac{2 \|\vecx\|_2}{\sqrt{n}} \, \expect \|S \vecxi\|_G \,.
		\nonumber
	\end{align}
	The proof is complete.
\end{proof}

We will need a complex version of the above result.

\begin{corollary}\label{cor:Ma91_l2G}
	Let $S:\ell_2^m = \C^m \to G$ be a $\C$-linear operator.
	Then the method from Proposition~\ref{prop:Ma91_l2G} provides the upper bound
	\begin{equation*}
		e^{\ran}(n,S: \ell_2^m = \C^m \to G)
			\,\leq\, \frac{2 \sqrt{2}}{\sqrt{n}} \, \expect \|S \vecxi\|_G
	\end{equation*}
	for $n<m$, where $\vecxi$ is a standard Gaussian vector in $\R^m$.
\end{corollary}
\begin{proof}
	Applying the algorithm $A_n$ from Proposition~\ref{prop:Ma91_l2G}
	to a vector~\mbox{$\vecx \in \C^m$},
	the error can be estimated by splitting the input~\mbox{$\vecx$} into
	its real and imaginary part~$\vecx = \Re \vecx + \imag \Im \vecx$.
	Using the triangle inequality we obtain
	\begin{align*}
		e(A_n,\vecx)
			&\leq e(A_n,\Re \vecx) + e(A_n, \Im \vecx) \\
			&\leq (\|\Re \vecx\|_2 + \|\Im \vecx\|_2) \, e(A_n,S:\R^m \to G) \\
			&\leq \sqrt{2} \, \|\vecx\|_2 \, e(A_n,S:\R^m \to G) \,.
	\end{align*}
	This gives the additional factor~$\sqrt{2}$ for the complex version.
\end{proof}

\begin{remark}[Basis representation]\label{rem:Ma91_basis}
	We will apply Proposition~\ref{prop:Ma91_l2G}
	or Corollary~\ref{cor:Ma91_l2G}, respectively,
	in a formally different situation.
	Suppose \mbox{$S: \Hilbert \to G$} is a linear rank~$m$ operator
	from a Hilbert space~$\Hilbert$ into a Banach space~$G$,
	and let~\mbox{$(\psi_j)_{j=1}^m$} be an orthonormal system in~$\Hilbert$
	such that
	\begin{equation*}
		S f \,=\, \sum_{j=1}^m \langle \psi_j, f \rangle_{\Hilbert} \, S \psi_j \,.
	\end{equation*}
	We use the randomized rank-$n$ method
	\begin{equation*}
		A_n^{\omega}(f)
			\,=\, \frac{1}{n} \sum_{i = 1}^n L_i^{\omega}(f) \, g_i^{\omega} \,,
	\end{equation*}
	where we choose random functionals~$L_i^{\omega}$,
	and corresponding functions~$g_i^{\omega}$ from the output space,
	\begin{equation*}
		L_i^{\omega}(f)
			\,:=\, \sum_{j=1}^m \xi_{ij} \, \langle \psi_j, f \rangle_{\Hilbert}\,,
		\qquad
		g_i^{\omega}
			\,:=\, \sum_{j=1}^m \xi_{ij} \, S \psi_j \,,
	\end{equation*}
	here the $\xi_{ij}$ are i.i.d.\ standard Gaussian random variables.
	According to the above results we have the error estimate
	\begin{equation*}
		e((A_n^{\omega})_{\omega},S)
			\,\leq\, C\, \frac{\expect \bigl\|\sum_{j=1}^m \xi_j \, S\psi_j \bigr\|_G
												}{\sqrt{n}} \,,
	\end{equation*}
	where~$\xi_j$ are standard real Gaussian random variables,
	and \mbox{$C = 2$} or \mbox{$2\sqrt{2}$}, respectively.
	Note that in the complex setting~$\K = \C$
	this estimate depends on the particular orthonormal system
	that we have chosen.
	Further, for function spaces~$\Hilbert$
	where we can canonically identify the real and the imaginary part,
	methods for the complex-numbered setting~\mbox{$\K = \C$}
	are not necessarily admissible for the real-numbered setting~\mbox{$\K = \R$}.
	For the situation we will consider, however,
	this is not a big problem, see Remark~\ref{rem:realAlg}.
\end{remark}

For some problems nonlinear Monte Carlo methods give better approximations.
The following abstract result is due to Heinrich 1992~\cite[Proposition~3]{He92}.

\begin{proposition} \label{prop:MCnonlin}
	Let $S: F \to H$ and $T: H \to G$ be bounded linear operators
	between Banach spaces. Then for $m,n \in \N_0$ we have
	\begin{equation*}
		e^{\ran,\nonlin}(m+n, T \, S)
			\,\leq\, e^{\det,\nonlin}(m,S) \, e^{\ran,\lin}(n,T) \,.
	\end{equation*}
\end{proposition}

The idea behind is a two-stage algorithm.
Specifically for a chain of embeddings~\mbox{$F \embed \Hilbert \embed G$},
within the first step, we use a nonlinear deterministic method~$B_m$
to give a rough approximation \mbox{$h = B_m(f)$},
which shall be close to the input~$f \in F$
with respect to the norm of an auxiliary Hilbert space~$\Hilbert$.
Within the second step, we compute an improved approximation
by applying a linear Monte Carlo method~$A_n^{\omega}$
(based on the method from Proposition~\ref{prop:Ma91_l2G})
to the residual~\mbox{$f-h$},
resulting in an output~\mbox{$g := h + A_n^{\omega}(f-h) \in G$}.

\section{Monte Carlo approximation of Sobolev classes in sup-norm}
\label{sec:main}

In this section we shall apply the result from the previous section
to investigate the asymptotic behaviour of the Monte Carlo error
for embeddings of periodic Sobolev spaces of dominating mixed smoothness
into $L_\infty(\tord)$.
We first give an introduction to these spaces. 

Let $\tord=[0,1)^d$ denote the $d$-torus. 
For $\vecs=(s_1,\ldots,s_d) \in \N_0^d$ and $j \in \N_0$ we put
\begin{equation} \label{eq:rho(s)}
	\begin{split}
		\rho(\vecs)
			&\,:=\, \bigl\{\veck \in \Z^d \,:\,
								\lfloor 2^{s_\ell-1} \rfloor
									\leq |k_\ell|
									< 2^{s_\ell},\ 
								\ell = 1,\ldots,d
					\bigr\} \,, \\
		Q_j &\,:=\, \bigcup_{|\vecs|_1 = j} \rho(\vecs) \,,
	\end{split}
\end{equation}
thus defining dyadic blocks~$\rho(\vecs)$ and hyperbolic layers~$Q_j$.
For a function~\mbox{$f\in L_1(\tord)$}
(being 1-periodic in every component)
the Fourier coefficients of $f$ are denoted by
\begin{equation*}
	c_{\veck}(f)
		\,:=\, \int_{\tord} f(\vecx) \, \euler^{- 2 \pi \imag (\veck,\vecx)} \dint \vecx
	,\qquad \veck\in \Z^d \,.
\end{equation*}
We put
\begin{equation*}
	[\delta_{\vecs} f](\vecx)
		\,:=\, \sum_{\veck \in \rho(\vecs)} c_k(f) \, \euler^{2\pi \imag (\veck,\vecx)}\,,
	\qquad \vecx \in \tord\,.
\end{equation*}

\begin{definition} \label{def:Wmix}
	Let $r \geq 0$ and $1 < p < \infty$.
	Then the periodic Sobolev spaces of dominating mixed smoothness $\W^r_p(\tord)$
	is the collection of all functions~\mbox{$f:\tord \to \C$}
	($1$-periodic in every component) such that
	\begin{equation*}
		\bigl\| f \bigr\|_{\W_p^r(\tord)}
			\,:=\, 
				\biggl\|\Bigl(\sum_{\vecs \in\N_0^d}
												2^{2 r |\vecs|_1}\, 
													\bigl| \delta_{\vecs} f \bigr|^2
								\Bigr)^{1/2}
				\biggr\|_{L_p(\tord)}
			\,<\, \infty\,.
	\end{equation*}
\end{definition}

\begin{remark}[Equivalent representations]
	Spaces of this type were studied systematically in the book
	of Schmeisser and Triebel~\cite{ST87}.
	They are the special case \mbox{$\W_p^r(\tord)=S^r_{p,2}F(\tord)$}
	of the Triebel-Lizorkin scale.
	Translated into the periodic setting,
	the authors of this book define a smooth resolution of unity
	to generate the building blocks~\mbox{$\delta_{\vecs} f$}.
	This is not necessary in our context dealing strictly with~\mbox{$1<p<\infty$}.
	A representation using cut-outs of the Fourier series
	by characteristic functions as given above
	is called 'Lizorkin representation', \mbox{cf.~\cite[Section~3.5.3]{ST87}}.

	If $r=0$ and $1<p<\infty$ we get back $L_p(\tord)$
	by the consequence of the Littlewood-Paley theorem.
	Note that if $r>1/p$,
	then the space $\W_p^r(\tord)$ is continuously embedded into $C(\tord)$,
	cf.~\cite[Section~2.4]{ST87}.
	For~$r \in \N$ and $1<p<\infty$ this norm is equivalent to
	the more ``natural'' norm
	\begin{equation}\label{eq:derivnorm}
		\|f\|_{\W_p^r(\tord)}'
			\,=\, \Biggl(\sum_{I \subseteq \{1,\ldots,d\}}
							\biggl\| \Bigl(\prod_{i \in I} \partial_i^r \Bigr) f
							\biggr\|_{L_p(\tord)}^p
						\Biggr)^{1/p} \,,
	\end{equation}
	see \cite[Section~2.3]{ST87}.
	There are several alternative approaches to a definition of a norm
	for these spaces, Fang and Duan~\cite{FD07,FD08}, for instance,
	chose a representation by convolutions with \mbox{$L_p$-functions},
	see Ullrich~\cite[Section~2.7]{Ul06} for the equivalence in the case $1<p<\infty$.
	For the critical cases~\mbox{$p \in \{1,\infty\}$}, unfortunately,
	these approaches to define a norm lead to different spaces.
\end{remark}

In the special case $p=2$, the space $\W_2^r(\tord)$ is a Hilbert space
and the norm given in Defintion~\ref{def:Wmix} can be written as
\begin{equation*}
	\|f\|_{\W_2^r(\tord)}
			\,=\, \biggl(\sum_{j=0}^{\infty} 2^{2jr} \sum_{k\in Q_j}|c_k(f)|^2
						\biggr)^{1/2} \,.
\end{equation*}
Note that the system
\begin{equation*}
	\psi_{\veck}(\vecx)
		\,:=\, 2^{-rj} \, \euler^{2 \pi \imag (\veck, \vecx)}
	\qquad\text{for}\;
	\veck \in Q_j \subset \Z^d,\; j \in \N_0,
\end{equation*}
forms an orthonormal basis of~$\W_2^r(\tord)$
with respect to the above norm.
Recognize the basis representation
of the fundamental Monte Carlo approximation method,
see Remark~\ref{rem:Ma91_basis},
within the definition of the following linear Monte Carlo algorithm for
the~$L_{\infty}$-approximation of functions from $\W_2^r(\tord)$. 

\begin{algorithm}\label{alg:MClin}
	For $J\in \N_0$ we put
	\begin{equation*}
		Q_{[J]}
			\,:=\, \bigcup_{j=0}^J Q_j
			\,=\, \bigcup_{0 \leq |\vecs|_1 \leq J}
					\rho(\vecs) \,.
	\end{equation*}
	Fourier coefficients from this index set
	will be collected directly,	that is,
	\mbox{$n_0 := |Q_{[J]}|$} pieces of information are used
	for the deterministic part of the algorithm.
	The same amount of linear information is spent
	on the collective approximation of Fourier coefficients
	from the hyperbolic layers~$Q_j$ for~\mbox{$j=J+1,\ldots,L$}
	via the fundamental Monte Carlo approximation method
	from Proposition~\ref{prop:Ma91_l2G}.
	
	In detail, we approximate~\mbox{$f \in \W_2^r(\Torus^d)$} by
	\begin{equation*}
		A_{J,L}^{\omega}(f)
			\,:=\, \sum_{\veck \in Q_{[J]}}
								c_{\veck}(f) \, \euler^{2 \pi \imag (\veck, \cdot)}
							+ \frac{1}{n_0} \, \sum_{i=1}^{n_0} L^{\omega}_i(f) \, g^{\omega}_i \,,
	\end{equation*}
	where we use random linear functionals and corresponding random functions
	defined by
	\begin{equation*}
		L^{\omega}_i(f)
			\,:=\, \sum_{j=J+1}^L
								2^{rj} \sum_{\veck \in Q_j}
													\xi_{i,\veck} \, c_{\veck}(f) \,,
		\qquad
		g^{\omega}_i(\vecx)
			\,:=\, \sum_{j=J+1}^L
						2^{-rj} \sum_{\veck \in Q_j}
											\xi_{i,\veck} \, \euler^{2 \pi \imag (\veck, \vecx)} \,,
	\end{equation*}
	here the~\mbox{$\xi_{i,\veck}$} are i.i.d.\ standard Gaussian random variables,
	$L\in \N$ will be chosen later on.

	Note that for~\mbox{$|\vecs|_1 = j$} we have~\mbox{$|\rho(\vecs)| = 2^j$}.
	Further,
	\begin{equation*}
		\frac{j^{d-1}}{(d-1)!}
			\,\leq\, \bigl|\{\vecs \in \N_0^d \,:\,
										|\vecs|_1 = j\}
								\bigr|
			\,=\, \binom{d+j-1}{d-1}
			\,\leq\, j^{d-1} \,.
	\end{equation*}
	Hence,
	\begin{equation} \label{eq:|Q_j|}
		|Q_j| \,\asymp\, 2^j \, j^{d-1} \,,
		\qquad\text{and}\qquad
		|Q_{[J]}|
			\,\asymp\, \sum_{j=0}^J 2^j \, j^{d-1}
			\,\asymp\, 2^J \, J^{d-1} \,.
	\end{equation}
	Then for the total information cost~$n = 2n_0$ we obtain
	\begin{equation}\label{eq:n><...}
		n \,\asymp\, 2^J \, J^{d-1} \,,
		\qquad \text{and} \qquad
		\log n \,\asymp\, J \,.
	\end{equation}
\end{algorithm}

\begin{remark}[Real-valued version of Algorithm~\ref{alg:MClin}] \label{rem:realAlg}
	Even when inserting a real-valued function $f \in \W_2^r(\tord)$,
	the output of the above algorithm in general will not be a real-valued function.
	Moreover, the functionals~$L^{\omega}_i$ may return complex numbers,
	which is not feasible in the real setting~$\K = \R$.
	Therefore we need a modified version of the algorithm
	based on
	strictly real-valued orthonormal basis functions.
	In detail, for nonzero~$\veck \in Q_j$ one should replace
	the pair of adjoint functions
	\mbox{$\psi_{\veck} = 2^{-rj} \, \euler^{2 \pi \imag (\veck,\cdot)}$}
	and \mbox{$\psi_{-\veck} = 2^{-rj} \, \euler^{-2 \pi \imag (\veck,\cdot)}$}
	by
	\begin{equation*}
		2^{-rj} \, \sqrt{2} \, \cos(2 \pi \imag (\veck,\cdot))
		\qquad \text{and} \qquad
		2^{-rj} \, \sqrt{2} \, \sin(2 \pi \imag (\veck,\cdot)) \,.
	\end{equation*}
	Using the basis representation from Remark~\ref{rem:Ma91_basis}
	with this modified basis,
	we obtain a valid algorithm for the real-valued setting,
	the corresponding error estimates will be similar.
\end{remark}

For the error analysis we need the following estimate on the expected
norm of a random trigonometric polynomial.
\begin{lemma} \label{lem:E|trigopol|}
	Let $E \subset \Z^d$ and define~\mbox{$\deg E := \max_{\veck \in E} |\veck|_1$},
	that is the largest degree of trigonometric polynomials that are
	composed of exponentials~$\euler^{2\pi\imag (\veck,\cdot)}$, \mbox{$\veck \in E$}.
	Then for~\mbox{$\deg E \geq 2$} we have
	\begin{equation*}
		\expect \biggl\| \sum_{\veck \in E}
											\xi_{\veck} \, \euler^{2 \pi \imag (\veck,\cdot)}
						\biggr\|_{L_{q}(\tord)}
			\,\preceq\,
				\begin{cases}
					\sqrt{q \, |E|} 						& \text{if }\; 2<q<\infty \\
					\sqrt{|E| \, \log (\deg E)}	& \text{if }\; q=\infty \,,
				\end{cases}
	\end{equation*}
	where $\xi_{\veck}$ are i.i.d.\ standard Gaussian random variables.
\end{lemma}
\begin{proof}
	Let~\mbox{$\vecnu \in \R^E$} be uniformly distributed
	on the euclidean unit \mbox{sphere}~\mbox{$\Sphere^{|E|-1} \subset \R^E$}.
	It has been proven by Belinsky~\cite[Lemma 3.4]{Be98} that
	\begin{equation*}\label{eq:belinsky}
		\expect \biggl\| \sum_{\veck \in E}
												\nu_{\veck} \, \euler^{2 \pi \imag (\veck,\cdot)}
							\biggr\|_{L_{q}(\tord)}
				\,\preceq\,
		\begin{cases} 
			\sqrt{q}						&\text{if }\; 2<q< \infty \,,\\
			\sqrt{\log(\deg E)}	&\text{if }\;  q=\infty \,.
		\end{cases}
	\end{equation*}
	Let \mbox{$\vecxi \in \R^E$} be a standard Gaussian vector.
	Due to the rotational invariance of the standard Gaussian measure,
	for computing expected values depending on~$\vecxi$
	we may switch to polar coordinates,
	in other words,
	\mbox{$\vecxi = \|\vecxi\|_2 \, (\vecxi / \|\vecxi\|_2)$},
	where~$\|\vecxi\|_2$ is independent from~\mbox{$\vecxi / \|\vecxi\|_2 \sim \vecnu$}.
	Exploiting the homogeneity of norms, we obtain
	\begin{equation*}
		\expect \biggl\| \sum_{\veck \in E}
											\xi_{\veck} \, \euler^{2 \pi \imag (\veck,\cdot)}
						\biggr\|_{L_{q}(\tord)}
			\,=\, \Biggl(\expect \biggl\| \sum_{\veck \in E}
																			\nu_{\veck} \, \euler^{2 \pi \imag (\veck,\cdot)}
														\biggr\|_{L_{q}(\tord)}
						\Biggr)
							\, \left(\expect \bigl\| \vecxi \bigr\|_{\ell_2^E} \right).
	\end{equation*}
	Employing $\expect \bigl\| \vecxi \bigr\|_{\ell_2^E}
							\leq \sqrt{\expect \bigl\| \vecxi \bigr\|_{\ell_2^E}^2}
							= \sqrt{|E|}$,
	and the result of Belinsky, we finish the proof.
\end{proof}

Via Algorithm~\ref{alg:MClin} we achieve the following estimate.
\begin{proposition} \label{prop:W2Linf}
	Let $r>1/2$. Then we have
	\begin{equation*}
		e^{\ran,\lin}\left(n, \W^r_2(\tord) \embed L_{\infty}(\tord)
									\right)
			\,\preceq\, \left(\frac{ (\log n)^{d-1}}{n}\right)^{r} \, \sqrt{\log n} \,,
	\end{equation*}
	where the implicit constant may depend on~$r$ and~$d$.
\end{proposition}

\begin{proof}
	We conduct the error analysis for the algorithm~$A_{J,L}^{\omega}$,
	at first assuming that our information budget is exactly what the algorithm needs,
	i.e.\ \mbox{$n = 2n_0 = 2|Q_{[J]}|$} and \eqref{eq:n><...} holds.
	
	Let~\mbox{$f \in \W^r_2(\tord)$} with
	\mbox{$\|f\|_{\W_2^r(\tord)} \leq 1$}.
	Decompose the input function into three parts,
	\begin{equation*}
		f = \Delta^J f + \Delta_J^L f + \Delta_L^{\infty} f \,,
	\end{equation*}
	where
	\begin{equation*}
		\Delta^J f
			\,:=\, \sum_{j=0}^{J} \sum_{|\vecs|_1 = j} \delta_{\vecs} f \,,
		\qquad
		\Delta_J^L f
			\,:=\, \sum_{j=J+1}^L \sum_{|\vecs|_1 = j} \delta_{\vecs} f \,,
		\qquad
		\Delta_L^{\infty} f
			\,:=\, \sum_{j=L+1}^{\infty} \sum_{|\vecs|_1 = j} \delta_{\vecs} f \,.
	\end{equation*}			
	The truncation parameter $L \in \N$ will be chosen later.
	Applying the algorithm~$A_{J,L}^\omega$ to~$f$,
	via the triangle inequality we obtain
	\begin{equation*}
		e(A_{J,L}^\omega,f)
			\,\leq\,
				0 + e(A_{J,L}^\omega, \Delta_J^L f)
					+ \|\Delta_L^{\infty} f\|_{L_{\infty}(\tord)} \,.
	\end{equation*}
	The first term vanishes
	since \mbox{$\Delta^J f$} is recovered exactly by~$A_{J,L}^\omega$.
	The second term can be estimated by Corollary~\ref{cor:Ma91_l2G},
	see also Remark~\ref{rem:Ma91_basis}.
	In doing so we need the following estimate on the expected norm
	with i.i.d.\ standard Gaussian random variables~\mbox{$\xi_{\veck}$},
	\begin{align*}
		\expect \biggl\|\sum_{j=J+1}^L
											2^{-rj} \sum_{\veck \in Q_j}
																\xi_{\veck} \, \euler^{2 \pi \imag (\veck, \cdot)}
						\biggr\|_{L_{\infty}(\tord)}
			&\,\leq\,
					\sum_{j=J+1}^L 2^{-rj}
						\expect \biggl\|\sum_{\veck \in Q_j}
															\xi_{\veck} \, \euler^{2 \pi \imag (\veck, \cdot)}
										\biggr\|_{L_{\infty}(\tord)} \\
		[\text{Lemma~\ref{lem:E|trigopol|}}]\qquad
			&\,\preceq\,
					\sum_{j=J+1}^L 2^{-rj}
						\sqrt{|Q_j| \log (2^j-1)} \\
			&\,\stackrel{\eqref{eq:|Q_j|}}{\asymp}\,
					\sum_{j=J+1}^L 2^{-(r-1/2)j} \, j^{d/2} \\
			&\,\asymp\, 2^{-(r-1/2) J} \, J^{d/2} \,.
	\end{align*}
	This leads to the error bound
	\begin{equation*}
		e\left(A_{J,L}, \Delta_J^L f\right)
			\,\preceq\,
				\frac{2^{-(r-1/2)J} \, J^{d/2}}{\sqrt{n_0}}
			\,\stackrel{\eqref{eq:n><...}}{\asymp}\,
				2^{-r J} \, \sqrt{J}
			\,\asymp\, \left(\frac{(\log n )^{d-1}}{n}\right)^r \sqrt{\log n}.
	\end{equation*}
	We obtain this result uniformly in $L$.
	To finish to proof we consider the third term.
	Using the Cauchy-Schwartz inequality,
	and having $\|f\|_{\W_2^r(\tord)} \leq 1$ in mind,
	we get
	\begin{align*}
		\bigl\| \Delta_L^{\infty} f \bigr\|_{\infty}
			&\,\leq\, \sum_{j = L+1}^{\infty} \sum_{\veck \in Q_j} |c_{\veck}(f)|
			\,\leq\, \biggl(\sum_{j = L+1}^{\infty} 2^{-2rj} |Q_j| \biggr)^{1/2}
									\bigl\| \Delta_L^{\infty} f \bigr\|_{\W_2^r(\tord)} \\
			&\,\preceq\, \biggl(\sum_{j = L+1}^{\infty} 2^{-(2r-1)j}\, j^{d-1}
									\biggr)^{1/2}\\
			&\,\asymp\,  2^{-(r-1/2) L} \, L^{(d-1)/2}\,.
	\end{align*}
	Choosing~$L$ sufficiently large (depending on $d$ and $r$),
	the truncation error is dominated	by the error estimate for the Monte Carlo part.

	
	Finally, concerning the quantity
	\mbox{$e^{\ran,\lin}\left(n, \W^r_2(\tord) \embed L_{\infty}(\tord)
											\right)$},
	note that for any arbitrary information budget~\mbox{$n \in \N$}
	we may choose an algorithm~$A_{J,L}^\omega$ with \mbox{$2 \, |Q_{[J]}| \leq n$}.
	The asymptotic relation~\eqref{eq:n><...} between~$J$ and $n$ still holds,
	with slightly worse constants though.
	Hence, we obtain the desired asymptotic order.
\end{proof}

\begin{remark}[Differences to previous research]
	One may employ Algorithm~\ref{alg:MClin}
	for similar calculations on the $L_q$-approximation,
	\mbox{$2 < q < \infty$},
	using the corresponding results from Lemma~\ref{lem:E|trigopol|}.
	That way one can reproduce
	the exact asymptotic order for the Monte Carlo approximation
	of~\mbox{$\W^r_2(\tord) \embed L_q(\tord)$}
	that has been determined
	by Fang and Duan~\cite[Theorem~1]{FD08}.
	The algorithm behind their estimates, however,
	is hidden within theory of pseudo $s$-scales.
	In particular, by conducting Maiorov's discretization technique,
	they apply the fundamental Monte Carlo method from Proposition~\ref{prop:Ma91_l2G}
	to the approximation of single blocks~\mbox{$\delta_{\vecs} f$},
	using a correspondence to sequence space embeddings.
	First of all, in the case of~$L_{\infty}$-approximation
	we lack a similar correspondence to sequence space embeddings,
	instead we took the estimate from Lemma~\ref{lem:E|trigopol|}.
	Second, the fundamental Monte Carlo method draws its strength
	from \emph{simultaniously} measuring \emph{all} Fourier coefficients that are to
	be approximated randomly, compare \cite[Section~3.2.3]{Ku17}.
\end{remark}

Our main result reads as follows.
\begin{theorem}\label{thm:main}
	(i)\; Let $1<p<\infty$ and $r>\max(1/p,1/2)$. Then we have
		\begin{equation*}
				\left(\frac{ (\log n)^{(d-1)}}{n}\right)^{r-(\frac{1}{p}-\frac{1}{2})_+}
				  \preceq\,
						  e^{\ran,\lin}\left(n,\W^r_p(\tord) \embed L_{\infty}(\tord)
															\right)
				  \,\preceq\,
							  \left(\frac{ (\log n)^{(d-1)}}{n}
									\right)^{r-(\frac{1}{p}-\frac{1}{2})_+}
										\sqrt{\log n} \,.
		\end{equation*}
	(ii)\; Let either $1<p<2$ and $r>1$, or $2\leq p<\infty$ and $r>1/2$.
		Then we have
		\begin{equation*}
				\left(\frac{ (\log n)^{(d-1)}}{n}\right)^r
					\ \preceq\ e^{\ran,\nonlin}\left(n,\W^r_p(\tord) \embed L_{\infty}(\tord)
																		\right)
					\ \preceq\ \left(\frac{ (\log n)^{(d-1)}}{n}
															\right)^r
																\sqrt{\log n} \,.
		\end{equation*}
	The implicit constants may depend on~$r$, $d$, and $p$.
\end{theorem}
\begin{proof} {\bf Estimate from below.\;}
	Note that~\mbox{$\|\cdot\|_{L_q(\tord)} \leq \|\cdot\|_{L_{\infty}(\tord)}$},
	hence $L_q$-approximation gives smaller errors
	than $L_{\infty}$-approximation.
	By this, the lower bounds follow from the asymptotic results for the embeddings
	\begin{equation*}
		\W_p^r(\tord)\embed L_q(\tord) \,, \qquad q>\max(p,2) \,,
	\end{equation*}
	that have been obtained by Fang and Duan
	for nonlinear Monte Carlo approximation~\cite{FD07},
	and for linear methods~\cite{FD08}, respectively.
	Note that the proofs for the lower bounds
	in the papers of Fang and Duan
	work for a larger range of the smoothness parameter~$r$
	than the corresponding upper bounds.
	
	{\bf Estimate from above.\;}
	This time we need the result in Proposition~\ref{prop:W2Linf}.
	If $p>2$,
	we use that~\mbox{$\|f\|_{\W_p^r(\tord)} \leq 1$}
	implies~\mbox{$\|f\|_{\W_2^r(\tord)} \leq 1$}
	and directly take the upper bound from Proposition~\ref{prop:W2Linf},
	compare the definition of the global error in~\eqref{eq:errglobal}.
	In the case $1<p<2$, we employ the Sobolev type embedding 
	\begin{equation} \label{eq:WpW2}
		\W_p^r(\tord) \embed \W_2^s(\tord) \,,
		\quad {\textstyle r \geq s + 1/p - 1/2} \,.
	\end{equation}
	For the linear setting,
	choosing~\mbox{$s := r - (1/p - 1/2)$},
	by \mbox{$r>1/p$} we guarantee \mbox{$s>1/2$},
	which is necessary for the embedding~\mbox{$\W_2^s(\tord) \embed L_{\infty}(\tord)$}
	to hold true and Proposition~\ref{prop:W2Linf} to be applicable.
	Then we obtain
	\begin{align*}
		e^{\ran,\lin}\left(n,\W_p^r(\tord) \embed L_{\infty}(\tord)\right)
			&\,\leq\, \|\W_p^r(\tord) \embed \W_2^{s}(\tord)\|
									\, e^{\ran,\lin}
											\left(n,\W_2^{s}(\tord)
																\embed L_{\infty}(\tord)
											\right)\\
			&\,\preceq\,
					\left(\frac{ (\log n)^{(d-1)}}{n}
					\right)^{r - (1/p - 1/2)}
							\sqrt{\log n}\,.
	\end{align*}
	In the nonlinear setting some improvement is possible.
	Regarding Definition~\ref{def:Wmix},
	we observe that the embedding~\eqref{eq:WpW2} is equivalent to the problem
	\begin{equation*}
		\W_p^{r-s}(\tord) \embed L_2(\tord) \,,
	\end{equation*}
	for which we have the estimate
	\begin{equation*}
		e^{\deter,\nonlin}\left(n,\W_p^{r-s}(\tord) \embed L_2(\tord)\right)
			\,\preceq\, \left(\frac{ (\log n)^{(d-1)}}{n} \right)^{r-s} \,,
			\quad {\textstyle r-s > \frac{1}{2}} \,,
	\end{equation*}
	see Fang and Duan~\cite[Theorem~2]{FD07} and the references therein.
	For any choice of the auxilieary parameter~$s$ such that \mbox{$1/2 < s < r - 1/2$},
	by Proposition~\ref{prop:MCnonlin} we obtain
	\begin{align*}
		e^{\ran,\nonlin}&\left(2 n,\W_p^r(\tord) \embed L_{\infty}(\tord)\right)\nonumber \\
			&\,\leq\, e^{\deter,\nonlin}\left(n,\W_p^r(\tord) \embed \W_2^{s}(\tord)\right)
								\, e^{\ran,\lin}\left(n,\W_2^{s}(\tord)
																	\embed L_{\infty}(\tord)
																\right) \\
			&\,\preceq\,
				\left(\frac{ (\log n)^{(d-1)}}{n} \right)^{r-s}
					\, \left(\frac{ (\log n)^{(d-1)}}{n} \right)^s \sqrt{\log n}\,,
	\end{align*}
	where $s$ cancels out.
	This is where we need the condition~$r > 1$.
	The shift from $n$ to $2n$ does not affect the rate of convergence,
	hence the proof is complete.
\end{proof}

\begin{remark}[On the range of the smoothness parameter~$r$]
	As mentioned within the above proof,
	the condition on the smoothness $r>1$ is not needed for the lower bounds
	of the nonlinear setting in Fang and Duan~\cite[Theorem~1]{FD07}.
	Actually, the exact asymptotic order via linear Monte Carlo methods
	for $L_q$-approximation in the case~\mbox{$2 \leq p < q < \infty$}
	is contained in their second paper~\cite[Theorem~1]{FD08},
	holding for~$r > 1/2$ already.
	Similarly, for~\mbox{$1 < p < q \leq 2$}
	we know results for nonlinear deterministic approximation that hold
	for smoothness~$r > 1/2$, see~\cite[Theorem~2]{FD07} and the references given therein.
	Hence, the conditions on the smoothness
	for the nonlinear Monte Carlo approximation rates
	in Fang and Duan~\cite[Theorem~1]{FD07}
	can be relaxed accordingly.
	Still, there is a gap, since the embedding
	\mbox{$\W_p^r(\tord) \embed L_q(\tord)$}, \mbox{$1 < p < q \leq \infty$},
	is compact (and therefore approximable) if~\mbox{$r > 1/p-1/q$}.
\end{remark}

\begin{remark}[Comparison with deterministic approximation]
	From \cite{VKN2} and Temlyakov~\cite{Tem93}
	we know that for~$1 < p \leq 2$ and~$r > 1$ we have
	\begin{equation}\label{eq-k1}
		e^{\deter}\left(n,\W^r_p(\tord) \embed L_{\infty}(\tord)\right)
			\,\asymp\,  \frac{ (\log n)^{(d-1)r}}{n^{r-1/2}}
									  \,.
	\end{equation}
	For \mbox{$2<p<\infty$}, \mbox{$r>1/p$},
	there is still a logarithmic gap in what is known about the asymptotic behaviour,
	in detail,
	\begin{equation}\label{eq-k2}
		\frac{ (\log n)^{(d-1)(r+1/2 - 1/p)}}{n^{r-1/p}} 
			\,\preceq\, e^{\deter}\left(n,\W^r_p(\tord) \embed L_{\infty}(\tord)\right) 
			\,\preceq\, \frac{ (\log n)^{(d-1)(r+1-2/p)}}{n^{r-1/p}}.
	\end{equation}
	The lower bound is obtained from the fact
	that $e^{\deter}$ is bounded from below by the so-called Weyl numbers,
	see~\cite{VKN1}.
	The upper bound is due to hyperbolic approximation.
	
	Comparing \eqref{eq-k1} and \eqref{eq-k2}
	with our result in Theorem~\ref{thm:main},
	we observe that randomization does help for all~$1 < p < \infty$,
	and nonlinear Monte Carlo methods are needed for the optimal rate
	in the case~$1 < p < 2$.
	The latter is particularly interesting
	since nonlinearity does only help when combined with randomization.

	Our results for~$L_{\infty}$-approximation
	fit to the general picture
	for $L_q$-approximation of the classes $\W_p^r(\tord)$,
	see Fang and Duan~\cite{FD07,FD08} for~$1 < q < \infty$,
	and also similar results for non-periodic
	isotropic spaces due to Heinrich~\cite{He92}.
	We have different (open) regions within the $(p,q)$-domain:
	\begin{description}
		\item[\quad$O:$\;]
			$e^{\ran,\nonlin}
				\asymp e^{\ran,\lin}
				\asymp e^{\deter,\nonlin}
				\asymp e^{\deter,\lin}$,\; linear deterministic methods suffice,
		\item[\quad$A:$\;]
			$e^{\ran,\nonlin}
				\asymp e^{\ran,\lin}
				\prec e^{\deter,\nonlin}
				\asymp e^{\deter,\lin}$,\;
					randomization helps,
		\item[\quad$B:$\;]
			$e^{\ran,\nonlin}
				\prec e^{\ran,\lin}
				\prec e^{\det,\nonlin}
				\asymp e^{\det,\lin}$,\; nonlinearity only helps with randomization,
		\item[\quad$C:$\;]
			$e^{\ran,\nonlin}
				\prec e^{\det,\nonlin}
				\prec e^{\ran,\lin}
				\asymp e^{\det,\lin}$,\; nonlinearity helps more than randomization,
		\item[\quad$D:$\;]
			$e^{\ran,\nonlin}
				\asymp e^{\deter,\nonlin}
				\prec e^{\ran,\lin}
				\asymp e^{\deter,\lin}$,\; nonlinearity helps.
	\end{description}
	Note that for $q = \infty$ (bold line in Fig.~1)
	there is a logarithmic gap between upper and lower Monte Carlo bounds,
	so then at the lower edge of~$A$ we only know
	\mbox{$e^{\ran,\nonlin} \preceq e^{\ran,\lin}$},
	though we expect asymptotic equality, compare Remark~\ref{rem:OP}.
	\begin{center}
	\begin{tikzpicture}[scale=2]
		
		\draw[->] (-0.3,0.0) -- (2.4,0.0) node[below] {$\frac{1}{p}$};
		\draw[->] (0.0,-.3) -- (0.0,2.4) node[left] {$\frac{1}{q}$};
		
		\draw (1.0,0.03) -- (1.0,-0.03) node [below] {$\frac{1}{2}$};
		\draw (0.03,1) -- (-0.03,1) node [left] {$\frac{1}{2}$};
		
		\node at (1.1,2.25) {$\W^r_p(\tord) \embed L_{q}(\tord)$};
		
		\draw (0,2) -- (2,2);
		\draw (1,1) -- (2,1);
		\draw (1,0) -- (1,1);
		\draw (1,1) -- (2,1);
		\draw (2,2) -- (2,0);
		\draw (0,0) -- (2,2);
		\draw (1,1) -- (2,0);
		\node at (0.67,0.33) {$A$};
		\node at (1.33,0.33) {$B$};
		\node at (1.67,0.67) {$C$};
		\node at (1.67,1.33) {$D$};
		\node at (0.67,1.33) {$O$}; 
		\draw (2,0.03) -- (2,-0.03) node [below] {$1$};
		\draw (0.03,2) -- (-0.03,2) node [left] {$1$};
		\draw (0,0) -- (2,0) [line width = 2.5pt];
		\node at (1,-0.5) {Fig.\ 1: 
			Regions of the $(p,q)$-domain with different
			behaviour of algorithmic features};
	\end{tikzpicture}
	\end{center}
\end{remark}

\begin{remark}[Open problems] \label{rem:OP}
	Fang and Duan~\cite{FD07,FD08} covered the cases
	\mbox{$1 < p,q < \infty$}. We added~$q = \infty$.
	For~\mbox{$p \in \{1,\infty\}$}
	the definition of the norm we use is critical
	since the Littlewood-Paley theorem does not hold.
	That means there is no consistency with derivatives bounded in~$L_p(\tord)$.
	For integer smoothness a useful definition of the limiting situations
	is provided by~\eqref{eq:derivnorm}.
	Approximation covering the case~\mbox{$p = 1$}
	is expected to need deterministic methods from sparse approximation,
	where the paper~\cite{DU17} on similar problems might give a hint.
	For~$p = \infty$ we expect deterministic methods
	to be optimal for any output space.
	
	In our results there is a gap of a factor~$\sqrt{\log n}$.
	We expect this to be a deficiency of the lower bound
	since this gap is closed for isotropic Sobolev spaces,
	which for~$d=1$ coincide with spaces of dominating mixed smoothness,
	see Heinrich~\cite{He92}.
\end{remark}

\begin{remark}[On the nonlinear algorithm for~$1 < p < 2$.]
	In the case of nonlinear algorithms,
	the deterministic method used for the first step
	of the two-stage method, approximating~\eqref{eq:WpW2},
	might need complete information on the most relevant Fourier coefficients,
	similar to the linear Monte Carlo method Algorithm~\ref{alg:MClin}.
	When implementing the two-stage method one should better spend
	a third of the information on obtaining the most relevant Fourier coefficients,
	the remaining two thirds of admissible information operations
	is then split equally into a collective deterministic approximation
	of subordered Fourier coefficients and a collective randomized refinement.
\end{remark}

\section*{Acknowledgements}

The authors thank the organizers of the conference
``IBC on the 70th anniversary of Henryk Woźniakowski''
where this project has been initiated.
G.B.~gratefully acknowledges support
by the German Research Foundation (DFG) Ul-403/2-1
and the Emmy-Noether programme, Ul-403/1-1.
R.J.K.~acknowledges support from the DFG Research Training Group 1523,
and the DFG-priority program 1324.


\renewcommand{\refname}{Bibliography}

\end{document}